\documentclass[smallextended]{svjour3}       
\smartqed

\usepackage{graphicx}
\usepackage{amsmath}
\usepackage{amsfonts}
\usepackage{amssymb}
\usepackage{cite}
\usepackage{hyperref}
\usepackage{enumerate}

\newtheorem{assumption}[theorem]{Assumption}

\newcommand{\Ocal}{{\mathcal O}}
\newcommand{\Vcal}{{\mathcal V}}
\newcommand{\Ebb}{\mathbb{E}}

\newcommand{\Nbb}{\mathbb{N}}

\newcommand{\Sbb}{\mathbb{S}}

\newcommand{\xbar}{{\bar{x}}}

\newcommand{\Euclid}{\mathbb{E}}
\newcommand{\paren}[1]{\left(#1\right)}
\newcommand{\klam}[1]{\left\{#1\right\}}

\renewcommand{\Re}{{\rm Re}}
\renewcommand{\equiv}{:=}

\newcommand{\ucmin}[2]{\underset{#2}{\mbox{minimize}}~#1}
\newcommand{\set}[2]{\left\{#1\,\left|\,#2\right.\right\}}

\newcommand{\mmap}[3]{#1:\,#2\rightrightarrows #3\,}
\newcommand{\ip}[2]{\left\langle #1, #2 \right\rangle}
\newcommand{\norm}[1]{\left\|#1\right\|}
\renewcommand{\Bbb}{\mathbb{B}}

\DeclareMathOperator{\Id}{Id}
\DeclareMathOperator{\id}{Id}
\DeclareMathOperator{\prox}{prox}
\DeclareMathOperator{\argmin}{argmin\,}
\DeclareMathOperator{\dist}{dist}
\DeclareMathOperator{\cone}{{cone}}

\DeclareMathOperator{\aff}{aff}
\DeclareMathOperator{\Fix}{\mathsf{Fix}\,}

\newcommand{\pncone}[1]{N^{\text{\rm prox}}_{#1}} 
\newcommand{\rpncone}[2]{N^{\text{\rm prox}}_{#1|#2}} 
\newcommand{\reli}{\ensuremath{\operatorname{ri}}}

\begin{document}

\title{A convergent relaxation of the Douglas-Rachford algorithm\thanks{The research leading to these results has received funding from the German-Israeli Foundation Grant G-1253-304.6 and the European Research Council under the European Union's Seventh Framework Programme (FP7/2007-2013)/ERC grant agreement N${}^{\circ}$ 339681.}
}

\author{Nguyen Hieu Thao}


\institute{Nguyen Hieu Thao\at
              Institut f\"ur Numerische und Angewandte Mathematik,
 			  Universit\"at G\"ottingen,
 			  37083 G\"ottingen, Germany \\
 			  Delft Center for Systems and Control,
 			  Delft University of Technology,
 			  2628CD Delft, The Netherlands\\
 	          \email{h.nguyen@math.uni-goettingen.de, h.t.nguyen-3@tudelft.nl}
}

\date{Received: date / Accepted: date}

\maketitle

\begin{abstract}
This paper proposes an algorithm for solving structured optimization problems, which covers both the backward-backward and the Douglas-Rachford algorithms as special cases, and analyzes its convergence.
The set of fixed points of the algorithm is characterized in several cases.
Convergence criteria of the algorithm in terms of general fixed point operators are established.
When applying to nonconvex feasibility including the inconsistent case, we prove local linear convergence results under mild assumptions on regularity of individual sets and of the collection of sets which need not intersect.
In this special case, we refine known linear convergence criteria for the Douglas-Rachford algorithm (DR).
As a consequence, for feasibility with one of the sets being affine, we establish criteria for linear and sublinear convergence of convex combinations of the alternating projection and the DR methods.
These results seem to be new.
We also demonstrate the seemingly improved numerical performance of this algorithm compared to the RAAR algorithm for both consistent and inconsistent sparse feasibility problems.


\end{abstract}

\section{Introduction}\label{intro}
Convergence analysis has been one of the central and very active applications of variational analysis and mathematical optimization.
Examples of recent contributions to the theory of the field that have initiated efficient programs of analysis are \cite{AspChaLuk16,AttBolRedSou10,BolSabTeb14,LukNguTam16}.
It is the common recipe emphasized in these and many other works that there are two key ingredients required in order to derive convergence of a numerical method 1) regularity of the individual functions or sets such as \textit{convexity} and \textit{averagedness},  and 2) regularity of families of functions or sets at their critical points such as \textit{transversality}, \textit{Kurdyka-\L ojasiewicz property} and \textit{metric subregularity}.
The question of convergence for a given method can therefore be reduced to checking regularity properties of the problem data.
There have been a considerable number of works studying the two ingredients of convergence analysis in order to provide sharper tools in various circumstances, especially in nonconvex cases, e.g., \cite{BauLukPhaWan13b,DruIofLew15,HesLuk13,KhaKruTha15,KruLukNgu16,KruLukNgu17,KruTha16,LewLukMal09,LewMal08,LukNguTam16,NolRon16,Pha16}.

This paper proposes an algorithm called $T_{\lambda}$, which covers both the backward - backward/alternating projection and DR algorithms as special cases of choosing the parameter $\lambda\in[0,1]$, and reports its convergence results.
When applied to feasibility for two sets with one of the sets being affine, $T_{\lambda}$ is a convex combination of the two fundamental projection methods.
On the other hand, $T_{\lambda}$ can be viewed as a relaxation of DR.
Motivation for relaxing the DR algorithm comes from its lack of stability for inconsistent feasibility.
This phenomenon has been observed for the \emph{phase retrieval problem} via the Fourier transform which is essentially inconsistent due to the reciprocal relationship between the spatial and frequency variables of the Fourier transform \cite{Luk05,Luk08}.
To address this issue, a relaxation of DR, often known as RAAR, was introduced and applied to phase retrieval problems by Luke in the aforementioned papers.
In the framework of feasibility, RAAR is described as convex combinations of the basic DR operator and one projector.
Preliminary numerical tests on the performance of $T_{\lambda}$ in comparison with RAAR look promising.
This has motivated the study of convergence analysis of $T_{\lambda}$ in this current work.

After introducing basic notation and proving preliminary results in Section \ref{s:notation and preliminary result}, we introduce $T_{\lambda}$ in terms of fixed point operators, discuss its fixed point set (Proposition \ref{p:Fix_point_oper}), and establish general convergence criteria for $T_{\lambda}$ (Theorem \ref{t:metric subreg convergence}) in Section \ref{s:T lambda as fixed point operator}.
We discuss $T_{\lambda}$ in the framework of feasibility problems in Section \ref{s:application to feasibility}.
The fixed point set of $T_{\lambda}$ is characterized for convex inconsistent feasibility (Proposition \ref{p:fixed point for inconsistent convex}).
For consistent feasibility we prove almost averagedness of $T_{\lambda}$ (Proposition \ref{p: aver_from_reg_for_T_lamb}) and metric subregularity of $T_{\lambda}-\id$ (Lemma \ref{l:assymptotic}) under regular properties of individual sets and of collections of sets, respectively.
The two ingredients are combined to obtain local linear convergence of $T_{\lambda}$ (Theorem \ref{t:linear convergence of T_lambda}).
Section \ref{s:numerical simulation} is devoted to demonstrate the improved numerical performance of $T_{\lambda}$ compared to the RAAR algorithm for both consistent and inconsistent feasibility problems.

\section{Notation and preliminary results}\label{s:notation and preliminary result}

Our notation is standard, c.f. \cite{DonRoc14,Mor06.1,VA}.
The setting throughout this paper is a finite dimensional Euclidean space $\mathbb{E}$.
The norm $\|\cdot\|$ denotes the Euclidean norm.
The open unit ball and the unit sphere in a Euclidean space are denoted $\Bbb$ and $\Sbb$, respectively.
$\Bbb_\delta(x)$ stands for the open ball with radius $\delta>0$ and center $x$.
The distance to a set ${A}\subset\mathbb{E}$ with respect to the bivariate function $\dist(\cdot, \cdot)$ is defined by
\begin{equation*}
\dist(\cdot, {A}) \colon \mathbb{E}\to\mathbb{R}\colon x\mapsto \inf_{y\in{A}}\dist(x,y).
\end{equation*}
We use the convention that the distance to the empty set is $+\infty$.
The set-valued mapping
\begin{equation*}
\mmap{P_{A}}{\mathbb{E}}{\mathbb{E}}\colon
x\mapsto \set{y\in {A}}{\dist(x,y)=\dist(x,{A})}
\end{equation*}
is the \emph{projector} on ${A}$.  An element $y\in P_{A}(x)$ is called a {\em projection}.
This exists for any closed set $ {A}\subset\mathbb{E}$, as can be deduced by the
continuity and the coercivity of the norm.  Note that the projector is not, in general, single-valued,
and indeed uniqueness of the projector defines a type of regularity of the set ${A}$:  local
uniqueness characterizes {\em prox-regularity} \cite{PolRocThi00}  while global uniqueness characterizes convexity \cite{Bunt}.
Closely related to the projector is the {\em prox} mapping corresponding to a function $f$ and a stepsize $\tau>0$ \cite{Moreau62}
\[
 \prox_{\tau,f}(x)\equiv \argmin_{y\in\Euclid}\klam{f(y)+\tfrac{1}{2\tau}\norm{y-x}^2}.
\]
When $f=\iota_{A}$ is the \emph{indicator function} of $A$, that is $\iota_{A}(x)=0$ if $x\in A$ and $\iota_{A}(x)=+\infty$ otherwise,  then $\prox_{\tau,\iota_{A}}=P_{A}$ for all $\tau>0$.
The value function corresponding to the prox mapping is known as the {\em Moreau} envelope, which we denote
by $e_{\tau,f}(x)\equiv \inf_{y\in\Euclid}\klam{f(y)+\tfrac{1}{2\tau}\norm{y-x}^2}$.  When $\tau=1$ and $f=\iota_{A}$
the Moreau envelope is just one-half the squared distance to the set ${A}$:
$e_{1,\iota_{A}}(x) = \tfrac12\dist^2(x,{A})$.
The \emph{inverse projector} $P^{-1}_{A}$ is defined  by
\begin{equation*}
 P^{-1}_{A} (y)\equiv\set{x\in\mathbb{E}}{y\in P_{A}(x)}.
\end{equation*}
The {\em proximal normal cone}  to ${A}$ at
$\xbar$ is the set, which need not be either closed or convex,
\begin{equation}\label{NC2}
\pncone{{A}}(\xbar)\equiv \cone\paren{P_{A}^{-1}\xbar-\xbar}.
\end{equation}
If $\xbar\notin{A} $, then $\pncone{{A}}(\xbar)$ is defined to be empty.
Normal cones are central to characterizations both of the regularity of individual sets and of the regularity of collections of sets.
For a refined numerical analysis of projection methods, one can define the \emph{$\Lambda$-proximal normal cone} to $A$ at $\bar{x}$
\begin{align*}
\rpncone{A}{\Lambda}(\xbar):=\cone\left((P_A^{-1}(\xbar)\cap \Lambda)-\xbar\right).
\end{align*}
When $\Lambda=\mathbb{E}$, it coincides with the proximal normal cone \eqref{NC2}.
We refer the reader to \cite{BauLukPhaWan13b} for a thorough discussion on the restricted versions of various normal cones.

For $\varepsilon\ge 0$ and $\delta>0$, a set $A$ is \emph{$(\varepsilon,\delta)$-regular relative to $\Lambda$} at $\bar{x}\in A$ \cite[Definition 2.9]{HesLuk13} if for all $x\in \mathbb{B}_{\delta}(\bar{x})$, $a\in A\cap\mathbb{B}_{\delta}(\bar{x})$ and $v \in \rpncone{A}{\Lambda}(a)$,
\begin{align*}
\ip{x-a}{v} \le \varepsilon\norm{x-a}\norm{v}.
\end{align*}
When $\Lambda=\Ebb$, the quantifier ``relative to'' is dropped.

For a set-valued operator $T:\mathbb{E}\rightrightarrows \mathbb{E}$, its \emph{fixed point set} is defined by
$\Fix T\equiv \set{x\in \Ebb}{x\in Tx}$.
We denote the {\em $\lambda$-reflector} of $T$ by $R_{T,\lambda}\equiv (1+\lambda)T-\lambda\id$.
A frequently used example in this paper corresponds to $T$ being a projector.

In the context of convergence analysis of Picard iterations, the following generalization of the Fej\'er monotonicity of sequences appears frequently, see, for example, the book \cite{BauCom11} or the paper \cite{LukNguTeb17} for the terminology.
\begin{definition}[linear monotonicity]\label{d:mu_Mon}
The sequence $(x_k)$ is \emph{linearly monotone} with respect to a set $S\subset \mathbb{E}$ with rate $c\in [0,1]$ if
\begin{equation*}
\dist(x_{k+1},S) \leq c\dist(x_k,S)\quad  \forall k\in \Nbb.
\end{equation*}
\end{definition}

Our analysis follows the abstract analysis program proposed in \cite{LukNguTam16} which requires the two key components of the convergence: \emph{almost averagedness} and \emph{metric subregularity}.

\begin{definition}[almost nonexpansive/averaged mappings]\label{d:ane-aa}\cite{LukNguTam16}
Let $T:\mathbb{E} \to \mathbb{E}$ and $U\subset \mathbb{E}$.
\begin{enumerate}[(i)]
\item $T$ is \emph{pointwise almost nonexpansive} on $U$ at $y\in U$ with violation $\varepsilon$ if for all $x\in U$, $x^+\in Tx$, and $y^+\in Ty$,
\begin{align*}
\norm{x^+- y^+} \le \sqrt{1+\varepsilon}\norm{x-y}.
\end{align*}
\item $T$ is \emph{pointwise almost averaged} on $U$ at $y\in U$ with violation $\varepsilon$ and averaging constant $\alpha>0$ if for all $x\in U$, $x^+\in Tx$, and $y^+\in Ty$,
\begin{align}\label{averaged of T}
\norm{x^+- y^+}^2 \le \paren{1+\varepsilon}\norm{x-y}^2 - \frac{1-\alpha}{\alpha}\norm{(x^+-x)-(y^+-y)}^2.
\end{align}
\end{enumerate}
When a property holds on $U$ at every point $z\in U$, we simply say the property holds on $U$.
\end{definition}

From Definition \ref{d:ane-aa}, the almost nonexpansive property is actually the almost averaged property with the same violation and averaging constant $\alpha=1$.

It is worth noting that if the iteration $x_{k+1}\in Tx_{k}$ is linearly monotone with respect to $\Fix T$ and $T$ is almost averaged, then $(x_k)$ converges R-linearly to a fixed point \cite[Proposition 3.5]{LukNguTeb17}.

We next prove a fundamental preliminary result for our analysis regarding almost averaged mappings.

\begin{lemma}\label{l:averaged of reflector}
Let $T:\mathbb{E}\rightrightarrows \mathbb{E}$, $U\subset \mathbb{E}$ and $\lambda\in [0,1]$.
The following statements are equivalent.
\begin{enumerate}[(i)]
  \item\label{l:averaged of reflector 1} $T$ is almost averaged on $U$ with violation $\varepsilon$ and averaging constant $\alpha\le \frac{1}{1+\lambda}$.
  \item\label{l:averaged of reflector 2} The $\lambda$-reflector $R_{T,\lambda}\equiv (1+\lambda)T - \lambda\id$ is almost averaged on $U$ with violation $(1+\lambda)\varepsilon$ and averaging constant
  \begin{equation}\label{alpha'}
      \alpha'\equiv \frac{1}{1+(1+\lambda)\paren{\frac{1-\alpha}{\alpha}-\lambda}}.
   \end{equation}
That is, for all $x,y\in U$, $\tilde{x}\in R_{T,\lambda}x$ and $\tilde{y}\in R_{T,\lambda}y$,
  \begin{align}\label{averaged of reflector}
  \norm{\tilde{x}-\tilde{y}}^2 \le (1+(1+\lambda)\varepsilon)\norm{x-y}^2 - 
  \frac{\frac{1-\alpha}{\alpha}-\lambda}{1+\lambda}\norm{(\tilde{x}-x)-(\tilde{y}-y)}^2.
  \end{align}
\end{enumerate}
\end{lemma}

\begin{proof} Let $x,y\in U$, $x^+\in Tx$, $y^+\in Ty$, $\tilde{x} = (1+\lambda)x^+-\lambda x \in R_{T,\lambda}x$ and $\tilde{y} = (1+\lambda)y^+-\lambda y \in R_{T,\lambda}y$.
We have by definition of $R_{T,\lambda}$ and \cite[Corollary 2.14]{BauCom11} that
\begin{align}\notag
        &\norm{\tilde{x}-\tilde{y}}^2
        \\\nonumber
        =\; &\norm{(1+\lambda)(x^+-y^+) -\lambda(x-y)}^2
        \\\label{axpand reflector}
        =\; &(1+\lambda)\norm{x^+-y^+}^2 - \lambda\norm{x-y}^2 + \lambda(1+\lambda)\norm{(x^+-y^+)-(x-y)}^2.
\end{align}
Substituting \eqref{averaged of T} into \eqref{axpand reflector} and noting that
\[
\norm{(\tilde{x}-x)-(\tilde{y}-y)} = (1+\lambda)\norm{(x^+-x)-(y^+-y)},
\]
we obtain
\begin{align*}
        &\norm{\tilde{x}-\tilde{y}}^2
        \\
        \le\;& (1+\varepsilon(1+\lambda))\norm{x-y}^2 - (1+\lambda)\paren{\frac{1-\alpha}{\alpha}-\lambda}\norm{(x^+-x)-(y^+-y)}^2
        \\
    =\; &(1+\varepsilon(1+\lambda))\norm{x-y}^2 - \frac{\frac{1-\alpha}{\alpha}-\lambda}{1+\lambda}\norm{(\tilde{x}-x)-(\tilde{y}-y)}^2,
\end{align*}
which is exactly \eqref{averaged of reflector}.

Conversely, substituting \eqref{averaged of reflector} into \eqref{axpand reflector}, we obtain
\begin{align*}
         & (1+\lambda)\norm{x^+-y^+}^2 - \lambda\norm{x-y}^2 + \lambda(1+\lambda)\norm{(x^+-y^+)-(x-y)}^2
        \\
         \le\;& (1+(1+\lambda)\varepsilon)\norm{x-y}^2 - \frac{\frac{1-\alpha}{\alpha}-\lambda}{1+\lambda}\norm{(\tilde{x}-x)-(\tilde{y}-y)}^2
         \\
        =\;& (1+(1+\lambda)\varepsilon)\norm{x-y}^2 - \paren{1+\lambda}\paren{\frac{1-\alpha}{\alpha}-\lambda}\norm{(x^+-x)-(y^+-y)}^2,
\end{align*}
which is equivalent to \eqref{averaged of T}. The proof is complete.
\hfill$\square$
\end{proof}

Lemma \ref{l:averaged of reflector} generalizes \cite[Lemma 2.4]{HesLuk13} where the result was proved for $\alpha=1/2$ and $\lambda=1$.
\bigskip

The next lemma recalls facts regarding the almost averagedness of projectors and reflectors associated with regular sets.

\begin{lemma}\label{l:lemma_aver_from_reg}
Let $A\subset \mathbb{E}$ be $(\varepsilon,\delta)$-regular at $\bar{x}\in A$ and define
\[
U\equiv\{x\in \mathbb{E}\mid P_Ax\subset \mathbb{B}_{\delta}(\bar{x})\}.
\]
\begin{enumerate}[(i)]
\item\label{l:lemma_aver_from_reg_1} The projector $P_A$ is pointwise almost nonexpansive on $U$ at every point $z\in A\cap \mathbb{B}_{\delta}(\bar{x})$ with violation $2\varepsilon+\varepsilon^2$.
\item\label{l:lemma_aver_from_reg_2} The projector $P_A$ is pointwise almost averaged on $U$ at every point $z\in A\cap \mathbb{B}_{\delta}(\bar{x})$ with violation $2\varepsilon+2\varepsilon^2$ and averaging constant $1/2$.
\item\label{l:lemma_aver_from_reg_3}
The $\lambda$-reflector $R_{P_A,\lambda}$ is pointwise almost averaged on $U$ at every point $z\in A\cap \mathbb{B}_{\delta}(\bar{x})$ with violation $(1+\lambda)(2\varepsilon+2\varepsilon^2)$ and averaging constant $\frac{1+\lambda}{2}$, that is, for all $x\in U$, $x^+\in R_{P_A,\lambda}x$ and $z\in A\cap \mathbb{B}_{\delta}(\bar{x})$,
\begin{align*}
\norm{x^+-{z}}^2 \le \paren{1+(1+\lambda)(2\varepsilon+2\varepsilon^2)}\norm{x-z}^2 - \frac{1-\lambda}{1+\lambda}\norm{x^+-x}^2.
\end{align*}
\end{enumerate}
\end{lemma}

\begin{proof}
Statements \eqref{l:lemma_aver_from_reg_1} and \eqref{l:lemma_aver_from_reg_2} were proved in \cite[Theorem 2.14]{HesLuk13}.
Statement \eqref{l:lemma_aver_from_reg_3} follows from \eqref{l:lemma_aver_from_reg_2} and Lemma \ref{l:averaged of reflector} applied to $T=P_A$ and $\alpha=1/2$.
\hfill$\square$
\end{proof}

The following concept of \emph{metric subregularity with functional modulus} has played a central role, implicitly and explicitly, in the analysis of convergence of Picard iterations \cite{AspChaLuk16,HesLuk13,LukNguTam16,LukNguTeb17}.
Recall that a function $\mu:[0,\infty) \to [0,\infty)$ is a \textit{gauge function} if $\mu$ is continuous and strictly increasing
with $\mu(0)=0$ and $\lim_{t\to \infty}\mu(t)=\infty$.

\begin{definition}[metric subregularity with functional modulus]\label{d:(str)metric (sub)reg}
A mapping $\mmap{F}{\mathbb{E}}{\mathbb{E}}$ is \emph{metrically subregular with gauge $\mu$ on $U\subset\mathbb{E}$ for $y$ relative to $\Lambda\subset \mathbb{E}$} if
\begin{equation}\label{e:metricregularity}
\mu\paren{\dist\paren{x,{F}^{-1}(y)\cap \Lambda}}\leq \dist\paren{y,{F}(x)}\quad \forall x\in U\cap \Lambda.
\end{equation}
When $\mu$ is a linear function, that is $\mu(t)=\kappa t,\, \forall t\in [0,\infty)$, one says ``with constant $\kappa$'' instead of
``with gauge $\mu=\kappa\id$''.
When $\Lambda={X}$, the quantifier ``relative to'' is dropped.
\end{definition}

Metric subregularity has many important applications in variational analysis and mathematical optimization, see the monographs and papers \cite{DonRoc14,Iof00,Iof11,Iof13,Iof16,KlaKum02,KlaKum09,Kru15,Pen13,Mor06.1}.
For the discussion of metric subregularity in connection with subtransversality of collections of sets, we refer the reader to \cite{Kru06,Kru09,KruTha14,KruTha15}.

The next theorem serves as the basic template for the quantitative convergence analysis of fixed point iterations.
By the notation $\mmap{T}{\Lambda}{\Lambda}$ where $\Lambda$ is a subset of $\mathbb{E}$, we mean that $\mmap{T}{\mathbb{E}}{\mathbb{E}}$ and $Tx\subset \Lambda$ for all $x\in \Lambda$.
This simplification of notation should not lead to any confusion if one keeps in mind
that there may exist fixed points of $T$ that are not in $\Lambda$.
For the importance of the use of $\Lambda$ in isolating the desirable fixed point, we refer the
reader to \cite[Example 1.8]{AspChaLuk16}.
In the following, $\reli \Lambda$ denotes the \emph{relative interior} of $\Lambda$.

\begin{theorem}\label{t:Tconv}\cite[Theorem 2.1]{LukNguTam16}
   Let $\mmap{T}{\Lambda}{\Lambda}$ for $\Lambda\subset\mathbb{E}$ and let $S\subset\reli \Lambda$ be closed and nonempty with
   $Ty\subset\Fix T\cap S$ for all
   $y\in S$.
   Let $\Ocal$ be a neighborhood of $S$ such that
   $\Ocal\cap \Lambda\subset \reli \Lambda$.
   Suppose that
\begin{enumerate}[(a)]
\item\label{t:Tconv a} T is pointwise almost averaged at all points $y\in S$
   with violation $\varepsilon$ and averaging constant $\alpha\in (0,1)$ on $\Ocal\cap \Lambda$, and
\item\label{t:Tconv b} there exists a neighborhood $\Vcal$ of $\Fix T\cap S$ and a constant $\kappa>0$
such that for all $y\in S$, $y^+\in Ty$ and all $x^+\in Tx$  the estimate
\begin{equation}\label{e:doughnut mreg}
\kappa\dist(x,S)\leq \norm{\paren{x-x^+}-\paren{y-y^+}}
\end{equation}
holds whenever  $x\in \paren{\Ocal\cap \Lambda}\setminus \paren{\Vcal\cap \Lambda}$.
\end{enumerate}
Then for all $x^+\in Tx$
\begin{equation*}
   \dist\paren{x^+,\Fix T\cap S}\leq\sqrt{1+\varepsilon - \frac{(1-\alpha)\kappa^2}{\alpha}}\dist(x,S)
\end{equation*}
whenever  $x\in \paren{\Ocal\cap \Lambda}\setminus \paren{\Vcal\cap \Lambda}$.

In particular, if $\kappa>\sqrt{\frac{\varepsilon\alpha}{1-\alpha}}$,  then
for any initial point $x_0\in \Ocal\cap \Lambda$  the iteration
$x_{k+1}\in Tx_k$ satisfies
\begin{equation*}
   \dist\paren{x_{k+1},\Fix T\cap S}\leq c^k \dist(x_0,S)
\end{equation*}
with $c\equiv \sqrt{1+\varepsilon-\frac{(1-\alpha)\kappa^2}{\alpha}}<1$
for all $k$ such that $x_j\in \paren{\Ocal\cap \Lambda}\setminus \paren{\Vcal\cap \Lambda}$ for $j=1,2,\dots,k$.
\end{theorem}

\begin{remark}\label{rem_metric_sub}
In the case of $S=\Fix T$ condition \eqref{e:doughnut mreg} reduces to metric subregularity of the mapping ${F}\equiv T-\Id$ for $0$ on the annular set, that is
\begin{equation*}
\kappa\dist(x,{F}^{-1}(0))\leq \dist(0,{F}(x))\quad \forall x\in \paren{\Ocal\cap \Lambda}\setminus \paren{\Vcal\cap \Lambda}.
\end{equation*}
The inequality $\kappa>\sqrt{\frac{\varepsilon\alpha}{1-\alpha}}$ then states that the constant of metric subregularity $\kappa$
is sufficiently large relative to the violation of the averaging property to guarantee a linear progression of the iterates through that annular region.
\end{remark}

For a comprehensive discussion on the roles of $S$ and $\Lambda$ in the analysis program of Theorem \ref{t:Tconv}, we would like to refer the reader to the paper \cite{LukNguTam16}.

For the sake of simplification in terms of presentation, we have chosen to reduce the number of technical constants appearing in the analysis.
It would be obviously analogous to formulate more theoretically general results by using more technical constants in appropriate places.

\section{$T_{\lambda}$ as a fixed point operator}\label{s:T lambda as fixed point operator}

We consider the problem of finding a fixed point of the operator
\begin{equation}\label{def T lambda}
T_{\lambda} \equiv T_1\paren{(1+\lambda)T_2 - \lambda\id} - \lambda\paren{T_2-\id},
\end{equation}
where $\lambda \in [0,1]$ and $T_i:\mathbb{E} \rightrightarrows \mathbb{E}$ $(i=1,2)$ are assumed to be easily computed.

Examples of $T_{\lambda}$ include the backward-backward and the DR algorithms \cite{BorTam15,ComPes11,LiPon16,Luk08,PatSteBem14} for solving the structured optimization problem
\begin{equation*}
   \ucmin{f_1(x)+f_2(x)}{x\in\mathbb{E}}
\end{equation*}
under different assumptions on the functions $f_i$ $(i=1,2)$.
Indeed, when $T_i$ are the prox mappings of $f_i$ with parameters $\tau_i>0$, then $T_{\lambda}$ with $\lambda=0$ and $1$ takes the form
$
T_{\lambda} = \prox_{\tau_1,f_1}\circ \prox_{\tau_2,f_2},
$
and
$
T_{\lambda} = \prox_{\tau_1,f_1}\paren{2\prox_{\tau_2,f_2}-\id} - \prox_{\tau_2,f_2} +\id
$,
respectively.
\bigskip

We first describe the fixed point set of $T_{\lambda}$ via those of the constituent operators $T_i$ $(i=1,2)$.

\begin{proposition}\label{p:Fix_point_oper} Let $T_1,T_2: \mathbb{E}\to \mathbb{E}$, $\lambda\in [0,1]$ and consider
\begin{equation}\label{RDR_general}
T_{\lambda} = T_1\left((1+\lambda)T_2-\lambda\id\right)-\lambda(T_2-\id).
\end{equation}
The following statements hold true.
\begin{enumerate}[(i)]
\item\label{p:Fix_point_oper_1} $(1+\lambda)T_{\lambda}-\lambda\id=\left((1+\lambda)T_1-\lambda\id\right)
\circ\left((1+\lambda)T_2-\lambda\id\right)$. As a consequence,
$$
\Fix T_{\lambda}=\Fix \left((1+\lambda)T_1-\lambda\id\right)
\circ\left((1+\lambda)T_2-\lambda\id\right).
$$
\item\label{p:Fix_point_oper_2} Suppose that $T_1=P_A$ is the projector associated with an affine set $A$ and $T_2$ is single-valued. Then
\begin{align*}
\Fix T_{\lambda}&=\{x\in \mathbb{E} \mid P_Ax =\lambda T_2x+(1-\lambda)x\}
\\
& \subset \{x\in \mathbb{E} \mid P_Ax = P_AT_2x\}.
\end{align*}
\end{enumerate}
\end{proposition}
\begin{proof}
\eqref{p:Fix_point_oper_1}.
\begin{align*}
(1+\lambda)T_{\lambda}-\lambda\id &=(1+\lambda)\left(T_1\left((1+\lambda)T_2-\lambda\id\right)
-\lambda(T_2-\id)\right)-\lambda\id
\\
&=(1+\lambda)T_1\left((1+\lambda)T_2-\lambda\id\right)-
\lambda\left[(1+\lambda)T_2-\lambda\id\right]
\\
&=\left((1+\lambda)T_1-\lambda\id\right)
\circ\left((1+\lambda)T_2-\lambda\id\right).
\end{align*}

\eqref{p:Fix_point_oper_2}.
\begin{align}
\notag
x = T_{\lambda}x = P_A\left((1+\lambda)T_2x-\lambda x\right)-\lambda(T_2x-x)
\\\label{tag1}
\Leftrightarrow \lambda T_2x+(1-\lambda)x=P_A\left((1+\lambda)T_2x-\lambda x\right).
\end{align}
In particular, $\lambda T_2x+(1-\lambda)x\in A$.
This together with equality \eqref{tag1} and $P_A$ being an affine operator imply
\begin{align}
\notag
P_A\paren{\lambda T_2x+(1-\lambda)x} &= P_A\left((1+\lambda)T_2x-\lambda x\right)
\\ \notag
\Leftrightarrow
\lambda P_AT_2x + (1-\lambda)P_Ax &= (1+\lambda)P_AT_2x-\lambda P_Ax
\\ \label{tag1'}
\Leftrightarrow
P_Ax &= P_AT_2x.
\end{align}
Substituting \eqref{tag1'} into \eqref{tag1} yields
\begin{align*}
\lambda T_2x+(1-\lambda)x &= (1+\lambda)P_AT_2x-\lambda P_Ax
\\
& = (1+\lambda)P_Ax-\lambda P_Ax = P_Ax.
\end{align*}
The proof is complete.
\hfill$\square$
\end{proof}

The next proposition shows that the almost averagedness of $T_{\lambda}$ naturally inherits from that of $T_1$ and $T_2$ via Krasnoselski--Mann relaxations.

\begin{proposition}[almost averagedness of $T_{\lambda}$]\label{alm_aver_T_lamb}
Let $T_i$ $(i=1,2)$ be almost averaged on $U_i\subset \mathbb{E}$ with violation $\varepsilon$ and averaging constant $\alpha$ and define
\[
U = \{x\in U_2 \mid R_{T_2,\lambda}x \subset U_1\}.
\]
Then $T_{\lambda}$ is almost averaged on $U$ with violation $2\varepsilon + (1+\lambda)\varepsilon^2$ and averaging constant
\begin{equation}\label{tilde alpha}
\tilde{\alpha} \equiv \frac{1}{(1+\lambda)\paren{1+\frac{1+\lambda}{2}\paren{\frac{1-\alpha}{\alpha}-\lambda}}}.
\end{equation}
\end{proposition}

\begin{proof}
By Lemma \ref{l:averaged of reflector}, the operators $R_{T_i,\lambda} := (1+\lambda)T_i - \lambda\id$ $(i=1,2)$ are almost averaged on $U_i$ with violation $(1+\lambda)\varepsilon$ and averaging constant $\alpha'$ given by \eqref{alpha'}.
Then thanks to \cite[Proposition 2.10 (iii)]{LukNguTam16}, the operator $T\equiv R_{T_1,\lambda}R_{T_2,\lambda}$ is almost averaged on $U$ with violation $(1+\lambda)\paren{2\varepsilon+(1+\lambda)\varepsilon}$ and averaging constant $\frac{2\alpha'}{1+\alpha'}$.
Note that $T_{\lambda}=(1+\lambda)T -\lambda\id$ by Proposition \ref{p:Fix_point_oper}.
We have, again by Lemma \ref{l:averaged of reflector}, that $T_{\lambda}$ is almost averaged on $U$ with violation $2\varepsilon+(1+\lambda)\varepsilon$ and averaging constant $\tilde{\alpha}$ given by \eqref{tilde alpha} as claimed.
\hfill$\square$
\end{proof}

We next discuss convergence of $T_{\lambda}$ based on the abstract results established in \cite{LukNguTam16}.
Our agenda is to verify the assumptions of Theorem \ref{t:Tconv}.
To simplify the exposure in terms of presentation, we have chosen to state the results corresponding to $S=\Fix T_{\lambda}$ and $\Lambda=\mathbb{E}$ in Theorem \ref{t:Tconv}.

\begin{theorem}[convergence of $T_{\lambda}$ with metric subregularity]\label{t:metric subreg convergence}
Let $T_{\lambda}$ be defined at \eqref{RDR_general} with $T_i$ $(i=1,2)$ being almost averaged with violation $\varepsilon$ and averaging constant $\alpha$.
Denote $S_\rho \equiv \Fix T_{\lambda} + \rho\Bbb$ for a nonnegative real $\rho$.
Suppose that there are $\delta>0$ and $\gamma\in (0,1)$ such that, for each $n\in \Nbb$, the mapping $F \equiv T_{\lambda}-\id$ is metrically subregular with gauge $\mu_n$ on $R_n\equiv S_{\gamma^n\delta}\setminus S_{\gamma^{n+1}\delta}$ for $0$, where $\mu_n$ satisfies
	\begin{equation}\label{e:kappa-epsilon}
\inf_{x\in R_n}
\frac{\mu_n\paren{\dist\paren{x,\Fix T_{\lambda}}}}{\dist\paren{x,\Fix T_{\lambda}}}
\ge\kappa_n > \sqrt{\frac{\tilde{\alpha}\varepsilon'}{1-\tilde{\alpha}}},
	\end{equation}
where $\tilde{\alpha}$ is given by \eqref{tilde alpha} and
\begin{align}\label{epsilon'}
\varepsilon'\equiv 2\varepsilon + (1+\lambda)\varepsilon^2.
\end{align}
Then, for any initial point $x_0\in S_{\delta}$,
the iterates $x_{k+1}\in T_{\lambda}x_k$ satisfy
\begin{align}\label{converge to zero}
\dist\paren{x_k,\Fix T_{\lambda}}\to 0,
\end{align}
and
\begin{equation}\label{e:metric subreg linear conv}
\dist\paren{x_{k+1}, \Fix T_{\lambda}} \leq c_n\dist\paren{x_k,\Fix T_{\lambda}} \quad\forall ~x_k\in R_n,
\end{equation}
where $c_n\equiv \sqrt{1+\varepsilon'-\frac{(1-\tilde{\alpha})\kappa_n^2}{\tilde{\alpha}}}<1$.

In particular, if $(\kappa_n)$ is bounded below by $\underline{\kappa}>\sqrt{\frac{\tilde{\alpha}\varepsilon'}{1-\tilde{\alpha}}}$ for all $n$ large enough, then $(x_k)$ is eventually R-linearly convergent to a fixed point with rate at most
$c\equiv \sqrt{1+\varepsilon'-\frac{(1-\tilde{\alpha})\underline{\kappa}^2}{\tilde{\alpha}}}<1$.
\end{theorem}

\begin{proof}
For each $n\in \mathbb{N}$, we verify the assumptions of Theorem \ref{t:Tconv} for $\Ocal=S_{\gamma^n\delta}$ and $\Vcal = S_{\gamma^{n+1}\delta}$.
Since $T_i$ $(i=1,2)$ are almost averaged with violation $\varepsilon$ and averaging constant $\alpha$, Proposition \ref{alm_aver_T_lamb} ensures that $T_{\lambda}$ is almost averaged with violation $\varepsilon'$ given by \eqref{epsilon'} and averaging constant $\tilde{\alpha}$ given by \eqref{tilde alpha}.
In other words, condition \eqref{t:Tconv a} of Theorem \ref{t:Tconv} is satisfied with $\varepsilon=\varepsilon'$ and $\alpha=\tilde{\alpha}$.
On the other hand, the metric subregularity of ${F}$ with gauge $\mu_n$ satisfying \eqref{e:kappa-epsilon} for $0$ on $R_n$ fulfills condition \eqref{t:Tconv b} of Theorem \ref{t:Tconv} with $\kappa=\kappa_n$ in view of Remark \ref{rem_metric_sub}.
Theorem \ref{t:Tconv} then yields the conclusion of Theorem \ref{t:metric subreg convergence} after a straightforward care of the involving technical constants.
\hfill$\square$
\end{proof}

The first inequality in \eqref{e:kappa-epsilon} essentially says that the gauge function $\mu_n$ can be bounded from below by a linear function on the reference interval.

\section{Applications to feasibility}\label{s:application to feasibility}

We consider $T_{\lambda}$ for solving feasibility involving two closed sets $A,B\subset \mathbb{E}$,
\begin{align}
\notag
x^+\in T_{\lambda}x &=P_A\left((1+\lambda)P_Bx-\lambda x\right)-\lambda\left(P_Bx-x\right)
\\\label{RDR}
&=P_AR_{P_B,\lambda}(x)-\lambda\left(P_Bx-x\right).
\end{align}
Note that $T_{\lambda}$ with $\lambda=0$ and $1$ corresponds to the alternating projection $P_AP_B$ and the DR methods $\frac{1}{2}(R_{A}\circ R_{B}+\id)$, respectively.
\bigskip

It is worth recalling that feasibility for $m\ge 2$ sets can be reformulated as feasibility for two constructed sets on the product space $\mathbb{E}^m$ with one of the later sets is a linear subspace, and the regularity properties in terms of both individual sets and collections of sets of the later sets are inherited from those of the former ones \cite{BauBor96,LewLukMal09}.

When $A$ is an affine set, then the projector $P_A$ is affine and $T_{\lambda}$ is the convex combination of the alternating projections and the DR since
\begin{align*}
T_{\lambda}x
&=
P_A\left((1-\lambda)P_Bx+\lambda(2P_Bx-x)\right)-\lambda\left(P_Bx-x\right)
\\
&=(1-\lambda)P_AP_Bx + \lambda\paren{x + P_A(2P_Bx-x) - P_Bx}
\\
&= (1-\lambda)T_0(x) + \lambda T_1(x).
\end{align*}

In this case, we establish convergence results for all convex combinations of the alternating projections and DR.
To our best awareness, this kind of results seems to be new.
\bigskip

Recall that for inconsistent feasibility the DR operator has no fixed points.
We next show that the fixed point set of $T_{\lambda}$ with $\lambda\in [0,1)$ for convex inconsistent feasibility is nonempty.
This result follows the lines of \cite[Lemma 2.1]{Luk08} where the fixed point set of the RAAR algorithm was characterized.

\begin{proposition}[Fixed points of $T_{\lambda}$ for convex inconsistent feasibility]\label{p:fixed point for inconsistent convex} For closed convex sets $A,B$, let $G=\overline{B-A}$, $g=P_G0$, $E=A\cap (B-g)$ and $F=(A+g)\cap B$. Then
$$
\Fix T_{\lambda}=E-\frac{\lambda}{1-\lambda}g\quad \forall t\in [0,1).
$$
\end{proposition}
\begin{proof}
We first show that $E-\frac{\lambda}{1-\lambda}g\subset \Fix T_{\lambda}$.
Pick any $e\in E$ and denote $f=e+g\in F$ as definitions of $E$ and $F$.
We are checking that 
\[
x\equiv e-\frac{\lambda}{1-\lambda}g\in \Fix T_{\lambda}.
\]
Since $x=f-\frac{1}{1-\lambda}g$ and $-g\in N_B(f)$, we get $P_Bx=f$.

Analogously, since $g\in N_A(e)$ and
\[
(1+\lambda)P_Bx-\lambda x=(1+\lambda)f-\lambda x=e+\frac{1}{1-\lambda}g,
\]
we have $P_A((1+\lambda)P_Bx-\lambda x)=e$.

Hence,
\begin{align*}
x-T_{\lambda}x=\; &x-P_A\left((1+\lambda)P_Bx-\lambda x\right)+\lambda\left(P_Bx-x\right)
\\
=\; &x-e+\lambda\left(f-x\right)=0.
\end{align*}
That is $x\in \Fix T_{\lambda}$.

We next show that
$
\Fix T_{\lambda}\subset E-\frac{\lambda}{1-\lambda}g
$.
Pick any $x\in \Fix T_{\lambda}$. Let $f=P_Bx$ and $y=x-f$.
Thanks to $x\in \Fix T_{\lambda}$ and the definition of $T_{\lambda}$,
\begin{align}\notag
P_A((1+\lambda)P_Bx-\lambda x)=\;&\lambda(P_Bx-x)+x
\\\label{2.9}
=\;&-\lambda y+y+f=f+(1-\lambda)y.
\end{align}
Now, for any $a\in A$, since $A$ is closed and convex, we have
\begin{align*}
0\ge\; & \ip{a-P_A((1+\lambda)P_Bx-\lambda x)}{(1+\lambda)P_Bx-\lambda x-P_A((1+\lambda)P_Bx-\lambda x)}
\\
=\;&\ip{a-(f+(1-\lambda)y)}{(1+\lambda)f-\lambda x-(f+(1-\lambda)y)}
\\
=\;&\ip{a-f-(1-\lambda)y}{-y} = \ip{-a+f}{y}+(1-\lambda)\norm{y}^2.
\end{align*}
On the other hand, for any $b\in B$, since $B$ is closed and convex, we have
$$
\ip{b-f}{y} = \ip{b-f}{x-f}=\ip{b-P_Bx}{x-P_Bx}\le 0.
$$
Combining the last two inequalities yields
$$
\ip{b-a}{y} \le -(1-\lambda)\norm{y}^2\le 0.
$$
Take a sequence $(a_n)$ in $A$ and a sequence $(b_n)$ in $B$ such that $g_n\equiv b_n-a_n\to g$. Then
\begin{equation}\label{2.17}
\ip{g_n}{y}\le -(1-\lambda)\norm{y}^2\le 0\quad \forall n.
\end{equation}
Taking the limit and using the Cauchy--Schwarz inequality yields
$$
\norm{y}\le \frac{1}{1-\lambda}\norm{g}.
$$
Conversely, by \eqref{2.9} with noting that $f\in B$ and $P_A((1+\lambda)P_Bx-\lambda x)\in A$,
$$
\norm{y}=\frac{1}{1-\lambda}\norm{f-P_A((1+\lambda)P_Bx-\lambda x)}\ge \frac{1}{1-\lambda}\norm{g}.
$$
Hence $\norm{y}=\frac{1}{1-\lambda}\norm{g}$, and taking the limit in \eqref{2.17}, which yields $y=-\frac{1}{1-\lambda}g$.
Since $f\in B$ and $f-g=f+(1-\lambda)y=P_A((1+\lambda)P_Bx-\lambda x)\in A$, we have $f-g\in A\cap (B-g)=E$ and, therefore,
$$
x=f+y=f-\frac{1}{1-\lambda}g=f-g-\frac{\lambda}{1-\lambda}g\in E-\frac{\lambda}{1-\lambda}g.
$$
\hfill$\square$
\end{proof}

We next discuss the two key ingredients for convergence of $T_{\lambda}$ when applied to feasibility problems: 1) almost averagedness of $T_{\lambda}$, and 2) metric subregularity of $T_{\lambda}-\id$.
The two properties will be verified under natural assumptions on $(\varepsilon,\delta)$-regularity of the sets and the transversality of the collection of sets, respectively.
\bigskip

The next proposition shows averagedness of $T_{\lambda}$ for feasibility involving $(\varepsilon,\delta)$-regular sets.

\begin{proposition}\label{p: aver_from_reg_for_T_lamb}
Let $A$ and $B$ be $(\varepsilon,\delta)$-regular at $\bar{x}\in A\cap B$ and define the set
\begin{equation}\label{U_definition}
U\equiv\{x\in \mathbb{E} \mid P_Bx\subset \mathbb{B}_{\delta}(\bar{x}) \mbox{ and } P_AR_{P_B,\lambda}x\subset \mathbb{B}_{\delta}(\bar{x})\}.
\end{equation}
Then $T_{\lambda}$ is pointwise almost averaged on $U$ at every point $z \in S$ with violation $\tilde{\varepsilon}$ and averaging constant $\frac{2}{3+\lambda}$, where $S\equiv A\cap B\cap\mathbb{B}_{\delta}(\bar{x})$ and
\begin{equation}\label{tilde_vareps}
\tilde{\varepsilon}\equiv2(2\varepsilon+2\varepsilon^2) + (1+\lambda)(2\varepsilon+2\varepsilon^2)^2.
\end{equation}
That is, for all $x\in U$, $x^+\in T_{\lambda}x$ and $z\in S$,
\begin{equation*}
\norm{x^+-z}^2 \le (1+\tilde{\varepsilon})\norm{x-z}^2 - \frac{1+\lambda}{2}\norm{x-x^+}^2.
\end{equation*}
\end{proposition}

\begin{proof} Let us define
\begin{align*}
U_A \equiv\{y\in \mathbb{E} \mid P_Ay\subset \mathbb{B}_{\delta}(\bar{x})\},\quad
U_B \equiv\{x\in \mathbb{E} \mid P_Bx\subset \mathbb{B}_{\delta}(\bar{x})\}
\end{align*}
and note that $x\in U$ if and only if $x\in U_B$ and $R_{P_B,\lambda}x\subset U_A$.
Thanks to Lemma \ref{l:lemma_aver_from_reg}~\eqref{l:lemma_aver_from_reg_3}, $R_{P_A,\lambda}$ and $R_{P_B,\lambda}$ are
pointwise almost averaged on $U$ at every point $z\in S$ with violation $(1+\lambda)(2\varepsilon+2\varepsilon^2)$ and averaging constant $\frac{1+\lambda}{2}$.
Then due to \cite[Proposition 2.10 (iii)]{LukNguTam16}, the operator $T\equiv R_{P_A,\lambda}R_{P_B,\lambda}$ is
pointwise almost averaged on $U$ at every point $z\in S$ with violation $(1+\lambda)\tilde{\varepsilon}$ and averaging constant $\frac{2(1+\lambda)}{3+\lambda}$, where $\tilde{\varepsilon}$ is given by \eqref{tilde_vareps}.
Note that $T_{\lambda}=(1+\lambda)T -\lambda\id$ by Proposition \ref{p:Fix_point_oper}.
Thanks to Lemma \ref{l:averaged of reflector}, $T_{\lambda}$ is pointwise almost averaged on $U$ at every point $z \in S$ with violation $\tilde{\varepsilon}$ and averaging constant $\frac{2}{3+\lambda}$ as claimed.
\hfill$\square$
\end{proof}

\begin{remark}\label{r:on U}
It follows from Lemma \ref{l:lemma_aver_from_reg}~\eqref{l:averaged of reflector 1} $\&$ \eqref{l:lemma_aver_from_reg_3} that the set $U$ defined by \eqref{U_definition} contains at least the ball $\mathbb{B}_{\delta'}(\bar{x})$, where
\begin{equation*}
\delta' \equiv \frac{\delta}{2(1+\varepsilon)\sqrt{1+(1+\lambda)(2\varepsilon+2\varepsilon^2)}}>0.
\end{equation*}
\end{remark}

We first integrate Proposition \ref{p: aver_from_reg_for_T_lamb} into Theorem \ref{t:metric subreg convergence} to obtain convergence of $T_{\lambda}$ for consistent feasibility involving regular sets.

\begin{corollary}[convergence of $T_{\lambda}$ for feasibility]\label{c:metric subreg convergence for feasibility}
Consider $T_{\lambda}$ for feasibility \eqref{RDR} and suppose that $\Fix T_{\lambda} = A\cap B \neq \emptyset$.
Denote $S_\rho \equiv \Fix T_{\lambda} + \rho\Bbb$ for a nonnegative real $\rho$.
Suppose that there are $\delta>0$, $\varepsilon\ge 0$ and $\gamma\in (0,1)$ such that $A$ and $B$ are $(\varepsilon,\delta')$-regular at avery point $z\in A\cap B$, where
\[
\delta'\equiv2\delta(1+\varepsilon)\sqrt{1+(1+\lambda)(2\varepsilon+2\varepsilon^2)},
\]
and for each $n\in \Nbb$, the mapping $F \equiv T_{\lambda}-\id$ is metrically subregular with gauge $\mu_n$ for $0$ on $R_n\equiv S_{\gamma^n\delta}\setminus S_{\gamma^{n+1}\delta}$, where $\mu_n$ satisfies
	\begin{equation*}
\inf_{x\in R_n}
\frac{\mu_n\paren{\dist\paren{x,A\cap B}}}{\dist\paren{x,A\cap B}}
\ge \kappa_n > \sqrt{\frac{2\tilde{\varepsilon}}{1+\lambda}},
	\end{equation*}
where $\tilde{\varepsilon}$ is given at \eqref{tilde_vareps}.

Then for any initial point $x_0\in S_{\delta}$,
the iterates $x_{k+1}\in T_{\lambda}x_k$ satisfy \eqref{converge to zero} and \eqref{e:metric subreg linear conv}
with $c_n\equiv \sqrt{1+\tilde{\varepsilon} - \frac{(1+\lambda)\kappa_n^2}{2}}<1$.

In particular, if $(\kappa_n)$ is bounded from below by $\underline{\kappa}>\sqrt{\frac{2\tilde{\varepsilon}}{1+\lambda}}$ for all $n$ large enough, then $(x_k)$ is eventually R-linearly convergent to a fixed point with rate at most
$c\equiv \sqrt{1+\tilde{\varepsilon} - \frac{(1+\lambda)\underline{\kappa}^2}{2}}<1$.
\end{corollary}

\begin{proof}
Let any $x\in R_n$, for some $n\in\mathbb{N}$, $x^+ \in T_{\lambda}x$ and $\bar{x}\in P_{A\cap B}x$.
Proposition \ref{p: aver_from_reg_for_T_lamb} and Remark \ref{r:on U} imply that $T_{\lambda}$ is pointwise almost averaged on $\mathbb{B}_{\delta}(\bar{x})$ at every point $z\in A\cap B\cap\mathbb{B}_{\delta}(\bar{x})$ with violation $\tilde{\varepsilon}$ given by \eqref{tilde_vareps} and averaging constant $\frac{2}{3+\lambda}$.
In other words, condition \eqref{t:Tconv a} of Theorem \ref{t:Tconv} is satisfied.
Condition \eqref{t:Tconv b} of Theorem \ref{t:Tconv} is also fulfilled for the same reason as that of Theorem \ref{t:metric subreg convergence}.
The desired conclusion now follows from Theorem \ref{t:Tconv}.
\hfill$\square$
\end{proof}

In practice, the metric subregularity assumption is often more challenging to verify than the averagedness one.
In the concrete example of consistent alternating projections $P_AP_B$, that condition holds true if and only if the collection of sets is subtransversal.
Our main goal in the rest of this study is to verify the metric subregularity of $T_{\lambda}-\id$ under the assumption of transversality.
As a result, if the sets are also sufficiently regular, then local linear convergence of the iterates $x_{k+1}\in T_{\lambda}x_k$ is guaranteed.
\bigskip

We first describe the concept of relative transversality of collections of sets.
In the sequel, we set $\Lambda \equiv \aff(A\cup B)$, the smallest affine set in $\mathbb{E}$ containing both $A$ and $B$.

\begin{assumption}\label{assumptions}
The collection $\{A,B\}$ is transversal at $\bar{x}\in A\cap B$ relative to $\Lambda$ with constant $\bar{\theta}<1$, that is, for any $\theta>\bar{\theta}$, there exists $\delta>0$ such that
        \[
        \ip{u}{v}\ge -\theta\norm{u}\cdot\norm{v},
        \]
 for all $a\in A\cap \mathbb{B}_{\delta}(\bar{x})$, $b\in B\cap \mathbb{B}_{\delta}(\bar{x})$, $u\in N_{A|\Lambda}^{\prox}(a)$ and
 $v\in N_{B|\Lambda}^{\prox}(b)$.
\end{assumption}

Thanks to \cite[Theorem 1]{Kru05} and \cite[Theorem 1]{KruTha13}, Assumption \ref{assumptions} also ensures subtransversality of $\{A,B\}$ at $\bar{x}$ relative to $\Lambda$ with constant at least $\sqrt{\frac{1-\theta}{2}}$ on the neighborhood $\mathbb{B}_{\delta}(\bar{x})$, that is
\begin{align}\label{assumption subtransversal}
\sqrt{\frac{1-\theta}{2}}\dist(x,A\cap B) \le \max\{\dist(x,A),\dist(x,B)\}\quad \forall x\in \Lambda\cap \mathbb{B}_{\delta}(\bar{x}).
\end{align}

The next lemma is at the heart of our subsequent discussion.

\begin{lemma}\label{l:assymptotic} Under Assumption \ref{assumptions}, for any $\theta\in (\bar\theta,1)$, there exists a number $\delta>0$ such that for all $x\in \mathbb{B}_{\delta}(\bar{x})$ and $x^+ \in T_{\lambda}x$,
\begin{equation}\label{assymptotic}
\kappa\dist(x,{A\cap B}) \le \norm{x-x^+},
\end{equation}
where $\kappa$ is defined by
\begin{align}\label{kappa}
\kappa \equiv \frac{(1-\theta)\sqrt{1+\theta}}{\sqrt{2}\max\left\{1,\lambda+\sqrt{1-\theta^2}\right\}}>0.
\end{align}
\end{lemma}

\begin{proof} Take any $x\in \mathbb{B}_{\delta}(\bar{x})$, $b\in P_Bx$, $y=(1+\lambda)b-\lambda x$, $a\in P_Ay$ and $x^+=a-\lambda(b-x)\in T_{\lambda}x$.
By choosing a smaller $\delta$ if necessary, we can assume without loss of generality that $a,b\in \mathbb{B}_{\delta}(\bar{x})$. For example, one can choose $\delta/6$ in place of $\delta$.

We first note that $x-b\in N_{B|\Lambda}^{\prox}(b)\; \mbox{ and }\; y-a\in N_{A|\Lambda}^{\prox}(a)$.
Assumption \ref{assumptions} then yields
\begin{align}\label{cosin condition}
\ip{x-b}{y-a}\ge -\theta\norm{x-b}\cdot\norm{y-a}.
\end{align}
By the definition of $T_{\lambda}$, we have
\begin{align}\notag
\norm{x-x^+}^2 &= \norm{x-b+y-a}^2
\\\notag
& = \norm{x-b}^2+\norm{y-a}^2+2\ip{x-b}{y-a}
\\\notag
&\geq \norm{x-b}^2+\norm{y-a}^2-2\theta\norm{x-b}\cdot\norm{y-a}
\\\label{estimate d(x,B)}
&\geq \paren{1-\theta^2}\norm{x-b}^2 = \paren{1-\theta^2}\dist^2(x,B),
\end{align}
where the first inequality follows from \eqref{cosin condition}.

We next consider the two cases regarding $\dist(x,A)$.

\emph{Case 1}. $\dist(x,A)\le \paren{\lambda+\sqrt{1-\theta^2}}\dist(x,B)$. Thanks to \eqref{estimate d(x,B)} we get
\begin{align}\label{estimate d(x,A) case 1}
\norm{x-x^+}^2 \geq \frac{1-\theta^2}{\paren{\lambda+\sqrt{1-\theta^2}}^2}\dist^2(x,A).
\end{align}

\emph{Case 2}. $\dist(x,A)> \paren{\lambda+\sqrt{1-\theta^2}}\dist(x,B)$. By the triangle inequality and the construction of $T_{\lambda}$, we get
\begin{align}\notag
\norm{x-x^+} &\geq \norm{x-a} - \norm{a-x^+} = \norm{x-a}-\lambda\norm{x-b}
 \\\label{estimate d(x,A) case 2}
& \ge \dist(x,A) - \lambda\dist(x,B)
\ge \paren{1-\frac{\lambda}{\lambda+\sqrt{1-\theta^2}}}\dist(x,A).
\end{align}
Since
\[
\frac{1-\theta^2}{\paren{\lambda+\sqrt{1-\theta^2}}^2} = \paren{1-\frac{\lambda}{\lambda+\sqrt{1-\theta^2}}}^2,
\]
we always have from \eqref{estimate d(x,A) case 1} and \eqref{estimate d(x,A) case 2} that
\begin{align}\label{estimate d(x,A)}
\norm{x-x^+}^2 \geq \frac{1-\theta^2}{\paren{\lambda+\sqrt{1-\theta^2}}^2}\dist^2(x,A).
\end{align}
Combining \eqref{estimate d(x,B)}, \eqref{estimate d(x,A)} and \eqref{assumption subtransversal}, we obtain
\begin{align*}
\norm{x-x^+}^2 & \geq \frac{1-\theta^2}{\max\left\{{1,\paren{\lambda+\sqrt{1-\theta^2}}^2}\right\}}\max\left\{\dist^2(x,A),\dist^2(x,B)\right\}
\\
&\geq \frac{(1-\theta^2)(1-\theta)}{2\max\left\{{1,\paren{\lambda+\sqrt{1-\theta^2}}^2}\right\}}\dist^2(x,A \cap B),
\end{align*}
which yields \eqref{assymptotic} as claimed.
\hfill$\square$
\end{proof}

In the special case of $\lambda=1$, Lemma \ref{l:assymptotic} improve \cite[Lemma 3.14]{HesLuk13} and \cite[Lemma 4.2]{Pha16} where the result was proved for DR with an additional assumption on regularity of the sets.
\bigskip

The next result is the final preparation for our linear convergence result.

\begin{lemma}\label{l:Phan convergence criterion}\cite[Proposition 2.11]{Pha16} Let $T:\mathbb{E}\rightrightarrows \mathbb{E}$, $S\subset \mathbb{E}$ be closed and $\bar{x}\in S$.
Suppose that there are $\delta>0$ and $c\in [0,1)$ such that for all $x\in \mathbb{B}_{\delta}(\bar{x})$, $x^+\in Tx$ and $z\in P_Sx$,
\begin{equation}\label{key estimate}
\norm{x^+-z}\le c\norm{x-z}.
\end{equation}
Then all iterates $x_{k+1}\in Tx_k$ converge R-linearly to a point $\tilde{x}\in S\cap \mathbb{B}_{\delta}(\bar{x})$ provided that $x_0$ is sufficiently close to $\bar{x}$. In particular,
\begin{equation*}
\norm{x_k-\tilde{x}} \le \frac{\norm{x_0-\bar{x}}(1+c)}{1-c}\,c^k.
\end{equation*}
\end{lemma}

We are now ready to prove local linear convergence of $T_{\lambda}$ which generalizes the corresponding results established in \cite{HesLuk13,Pha16} for DR.

\begin{theorem}[linear convergence of $T_{\lambda}$]\label{t:linear convergence of T_lambda} In addition to Assumption \ref{assumptions}, suppose that $A$ and $B$ are $(\varepsilon,\delta)$-regular at $\bar{x}$ with $\tilde{\varepsilon} < \frac{(1+\lambda)\kappa^2}{2}$, where $\tilde{\varepsilon}$ and $\kappa$ are given by \eqref{tilde_vareps} and \eqref{kappa}, respectively.
Then the iteration $x_{k+1} \in T_{\lambda}x_k$ converges R-linearly to a point in $A\cap B$ provided that $x_0$ is sufficiently close to $\bar{x}$.
\end{theorem}

\begin{proof}
Assumption \ref{assumptions} yields the existence of $\delta_1>0$ such that Lemma \ref{l:assymptotic} holds true.
In view of Proposition \ref{p: aver_from_reg_for_T_lamb} and Remark \ref{r:on U}, one can find a number $\delta_2>0$ such that $T_{\lambda}$ is pointwise almost averaged on $\mathbb{B}_{\delta_2}(\bar{x})$ at every point $z\in A\cap B\cap\mathbb{B}_{\delta_2}(\bar{x})$ with violation $\tilde{\varepsilon}$ given by \eqref{tilde_vareps} and averaging constant $\frac{2}{3+\lambda}$.
Let $\delta'\equiv\min\{\delta_1,\delta_2\}>0$.

Now we consider any $x\in \mathbb{B}_{\delta'}(\bar{x})$, $x^+\in T_{\lambda}x$ and $z\in P_{A\cap B}x$.
Again, by choosing a smaller $\delta$ if necessary, we can assume without loss of generality that $z\in \mathbb{B}_{\delta'}(\bar{x})$.
Proposition \ref{p: aver_from_reg_for_T_lamb} and Lemma \ref{l:assymptotic} then  respectively yield
\begin{align}\label{ingredient 1}
\norm{x^+-z}^2 &\le (1+\tilde{\varepsilon})\norm{x-z}^2 - \frac{1+\lambda}{2}\norm{x-x^+}^2,
\\\label{ingredient 2}
\norm{x-x^+}^2 &\ge \kappa^2\dist^2(x,{A\cap B}) = \kappa^2\norm{x-z}^2,
\end{align}
where $\kappa$ is given by \eqref{kappa}.

Substituting \eqref{ingredient 2} into \eqref{ingredient 1}, we get
\begin{align}\label{ingredient 3}
\norm{x^+-z}^2 &\le \paren{1+\tilde{\varepsilon}-\frac{(1+\lambda)\kappa^2}{2}}\norm{x-z}^2,
\end{align}
which yields condition \eqref{key estimate} of Lemma \ref{l:Phan convergence criterion} and the desired conclusion now follows from aforementioned lemma.
\hfill$\square$
\end{proof}

\section{Numerical experiments}\label{s:numerical simulation}

Our goal in this section is to demonstrate a promising performance of $T_{\lambda}$ in comparison with the RAAR algorithm for the example of sparse feasibility problem.
We first recall the \emph{sparse optimization problem} of
\begin{equation}\label{sparse optimization problem}
\min_{x\in \mathbb{R}^n} \norm{x}_0 \quad \mbox{ subject to }\quad Mx = b,
\end{equation}
where $M\in \mathbb{R}^{m\times n}$ $(m<n)$ is a full rank matrix, $b$ is a vector in $\mathbb{R}^m$, and $\norm{x}_0$ is the number of nonzero entries of the vector $x$.
The sparse optimization problem with complex variable is defined analogously by replacing $\mathbb{R}$ by $\mathbb{C}$ everywhere in the above model.

Many strategies for solving \eqref{sparse optimization problem} have been studied.
We refer the reader to the famous paper by Cand{\`e}s and Tao \cite{CanTao05} for solving this problem by using convex relaxations.
On the other hand, assuming to have a good guess on the sparsity of the solutions to \eqref{sparse optimization problem}, one can tackle this problem by solving the \emph{sparse feasibility problem}  \cite{HesLukNeu14} of finding
\begin{equation}\label{sparse feasibility problem}
\bar{x} \in A_s \cap B,
\end{equation}
where $A_s \equiv \{x\in \mathbb{R}^n\mid \norm{x}_0\le s\}$ and $B \equiv \{x\in \mathbb{R}^n\mid Mx = b\}$.

It is worth mentioning that the initial guess $s$ of the true sparsity is not numerically sensitive with respect to various projection methods, that is, for a relatively wide range of values of $s$ above the true sparsity, projection algorithms perform very much in the same nature.
Note also that the approach via sparse feasibility does not require convex relaxations of \eqref{sparse optimization problem} and thus can avoid the likely expensive increase of dimensionality.

We run $T_{\lambda}$ and RAAR to solve \eqref{sparse feasibility problem} and compare their numerical performances.
By taking $s$ smaller than the true sparsity, we can also compare their performances for inconsistent feasibility.

Since $B$ is affine, there is the closed algebraic form for the projector $P_B$,
\[
P_Bx = x - M^{\dagger}(Mx-b)\quad \forall x\in \mathbb{R}^n,
\]
where $M^{\dagger}:= M^T(MM^T)^{-1}$ is the \emph{Moore--Penrose inverse} of $M$.
We have denoted $M^T$ the transpose matrix of $M$ and take into account of full rank of $M$.
There is also a closed form for $P_{A_s}$ \cite{BauLukPhaWan14}.
For each $x\in \mathbb{R}^n$, let us denote $\mathcal{I}_x$ the set of all $s$-tubles of indices of $s$ largest in absolute value entries  of $x$.
The set $\mathcal{I}_x$ can contains multiple $s$-tubles since $A_s$ is intrinsically nonconvex.
The projector $P_{A_s}$ can be described by
\[
P_{A_s}x = \left\{z\in \mathbb{R}^n \mid \exists\; \mathcal{I}\in \mathcal{I}_x\; \mbox{ such that }
\; z(k)=
\begin{cases}
x(k) & \mbox{if } k \in \mathcal{I},
\\
0  & \mbox{else}
\end{cases}
\right\}.
\]
For convenience, we recall the two algorithms in this case
\begin{align*}
RAAR_\beta &= \beta\paren{P_{A_s}(2P_B-\id)} + (1-2\beta)P_B + \beta\id,
\\
T_{\lambda} &= P_{A_s}\paren{(1+\lambda)P_B-\lambda\id} - \lambda(P_B-\id).
\end{align*}

We now set up a toy example as in \cite{CanTao05,HesLukNeu14} which involves an unknown true object $\bar{x} \in \mathbb{R}^{256^2}$ with $\norm{\bar{x}}_0=328$ (the sparsity rate is $.005$).
Let $b$ be $1/8$ of the measurements of $F(\bar{x})$, the \emph{Fourier transform} of $\bar{x}$, with the \emph{sample indices} denoted by $\mathcal{J}$.
The \emph{Poisson noise} was added when calculating the measurements of $b$.
Note that since $\bar{x}$ is real, $F(\bar{x})$ is \emph{conjugate symmetric}, we indeed have nearly a double number of measurements.
In this setting, we have
\[
B=\{x\in \mathbb{C}^{256^2} : F(x)(k) = b(k), \; \forall k\in \mathcal{J}\},
\]
and the two prox operators, respectively, take the forms
\begin{align*}
P_{A_s}x &= \left\{z\in \mathbb{R}^n \mid \exists\; \mathcal{I}\in \mathcal{I}_x\; \mbox{ such that }
\; z(k)=
\begin{cases}
\Re\paren{x(k)} & \mbox{if } k \in \mathcal{I},
\\
0  & \mbox{else}
\end{cases}
\right\},
\\
P_Bx &= F^{-1}(\hat{x}),\; \mbox{ where }\; \hat{x}(k) =
\begin{cases}
b(k) & \mbox{if } k \in \mathcal{J},
\\
F(x)(k)  & \mbox{else,}
\end{cases}
\end{align*}
where $\Re(x(k))$ denotes the real part of the complex number $x(k)$, and $F^{-1}$ is the \emph{inverse Fourier transform}.

The initial point was chosen randomly, and a warm up procedure with around ten DR iterates was made before running the two algorithms.
The stopping criterion $\norm{x-x^+} < 10^{-10}$ was used.
We have used the \emph{Matlab ProxToolbox} \cite{Luk_PTB17} to run this experiment.
The parameters were chosen in such a way that the performance is seemingly the best for each algorithm.
We chose $\beta=.65$ for RAAR and $\lambda=.45$ for $T_{\lambda}$ in the consistent case corresponding to $s=340$,
and $\beta=.6$ for RAAR and $\lambda=.4$ for $T_{\lambda}$ in the inconsistent case corresponding to $s=310$.

The \emph{change} of distances between two consecutive iterates is of interest.
When linear convergence appears to be the case, it can yield useful information of the convergence rate.
Under the assumption that the iterates will remain in the convergence area, one can obtain error bounds for the distance from the iterate to a solution.
We also pay attention to the iterate \emph{gaps} that in a sense measure the infeasibility at the iterates.
If we think feasibility as the problem of minimizing the function that is the sum of the square of the distance functions to the sets, then iterate gaps are the values of that function evaluated at the iterates.
For the two algorithms under consideration, the iterates are themselves not informative but their shadows, by which we mean the projections of the iterates on one of the sets.
Hence, the iterate gaps are calculated for the iterate shadows instead of the iterates themselves.

Figure \ref{F:comparison for sparse feasibility} summarizes the performances of the two algorithms for both consistent and inconsistent sparse feasibility.
We first emphasize that the algorithms appear to be convergent in both cases of feasibility.
For the consistent case, $T_{\lambda}$ appears to perform better than RAAR in both the iterate changes and gaps.
Also, the CPU time for $T_{\lambda}$ is around $10\%$ less than that for RAAR.
For the inconsistent case, we have a similar observation except that the iterate gaps for RAAR are slightly better (smaller) than those for $T_{\lambda}$. Extensive numerical experiments in imaging illustrating the empirical performance of $T_{\lambda}$ will be the tasks for future work.

\begin{figure}
\begin{center}
\includegraphics[scale=1]{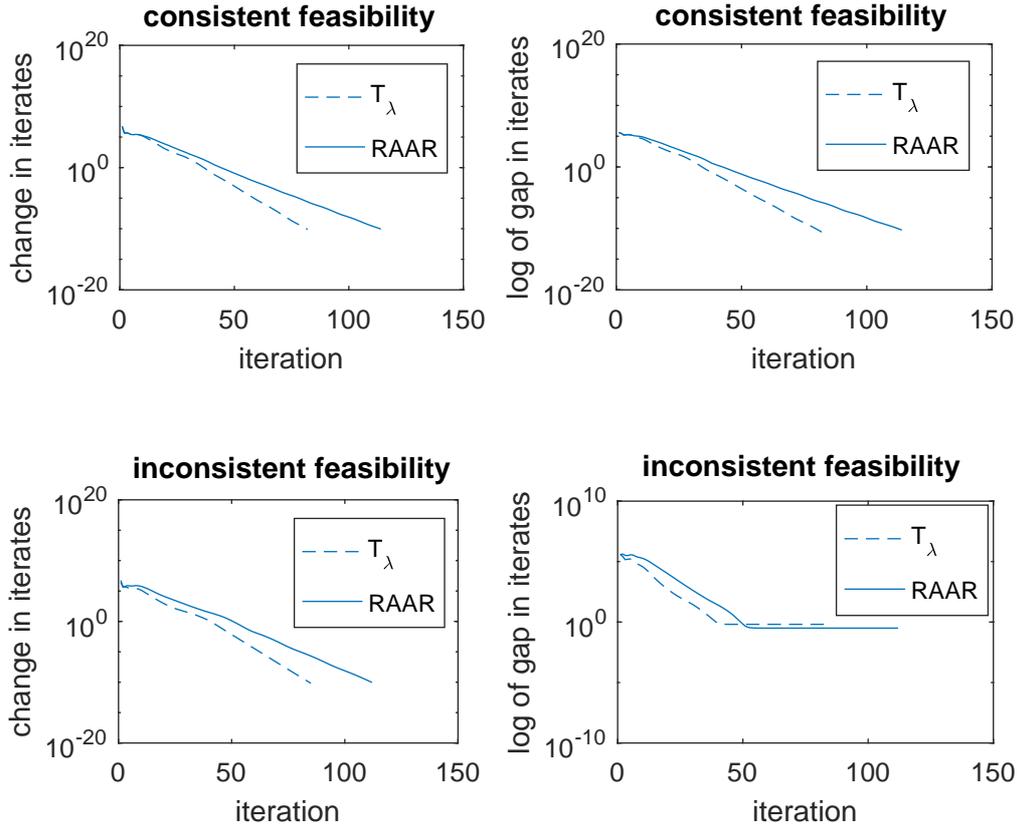}
\end{center}
\caption{Performances of RAAR and $T_{\lambda}$ for sparse feasibility problem: iterate changes in consistent case (top-left), iterate gaps in consistent case (top-right), iterate changes in inconsistent case (bottom-left) and iterate gaps in consistent case (bottom-right).}
\label{F:comparison for sparse feasibility}
\end{figure}

\begin{acknowledgements}
The author would like to thank Prof. Dr. Russell Luke and Prof. Dr. Alexander Kruger for their encouragement and valuable suggestions during the preparation of this work.
\end{acknowledgements}

\end{document}